\documentclass[12pt]{amsart}
\usepackage{amssymb,amsmath,amsfonts,amsthm,mathrsfs}

\usepackage[dvipdfmx]{graphicx,color}

\usepackage{hyperref}

\numberwithin{equation}{section}  
\theoremstyle{plain}
\newtheorem{thm}{Theorem}[section]

\newtheorem{cor}[thm]{Corollary}
\newtheorem{lem}{Lemma}[section]

\theoremstyle{definition}
\newtheorem{defn}[thm]{Definition}

\theoremstyle{remark}

\def\R{{\mathbb R}}
\def\C{{\mathbb C}}

\def\L2tx{{L^2(\R_t\times\R^n_x)}}

\def\p#1{{\left({#1}\right)}}
\def\b#1{{\left\{{#1}\right\}}}

\def\n#1{{\left\|{#1}\right\|}}
\def\abs#1{{\left|{#1}\right|}}
\def\jp#1{{\left\langle{#1}\right\rangle}}

\def\supp{\operatorname{supp}}

\def\A{{\mathcal A}}

\title[]
{A local-to-global boundedness argument \\ and Fourier integral operators}

\author[Michael Ruzhansky and Mitsuru Sugimoto]{Michael Ruzhansky and Mitsuru Sugimoto}
\address{
  Michael Ruzhansky:
  \endgraf
  Department of Mathematics
  \endgraf
  Imperial College London
  \endgraf
  180 Queen's Gate, London SW7 2AZ, UK
  \endgraf
  {\it E-mail address} {\rm m.ruzhansky@imperial.ac.uk}
  \endgraf
  \medskip
 Mitsuru Sugimoto:
  \endgraf
  Graduate School of Mathematics
  \endgraf
  Nagoya University
  \endgraf
  Furocho, Chikusa-ku, Nagoya 464-8602, Japan
  \endgraf
  {\it E-mail address} {\rm sugimoto@math.nagoya-u.ac.jp}
  }
\thanks{The first
 author was supported in parts by the EPSRC
 grants EP/K039407/1 and EP/R003025/1, and by the Leverhulme Grant RPG-2017-151.
The second author was supported in parts
by the JSPS KAKENHI 26287022 and 26610021.
}

\date{\today}
\dedicatory{Dedicated to the memory of Professor Hans Duistermaat (1942--2010)}

\begin{document}

\begin{abstract}
We give a criterion for the global boundedness of integral operators which are known to be locally bounded. As an application, we discuss the global $L^p$-boundedness for a class of Fourier integral operators. While the local $L^p$-boundedness of Fourier integral operators is known from the work of Seeger, Sogge and Stein \cite{SSS}, not so many results are available for the global boundedness on $L^p(\R^n)$. We give several natural sufficient conditions for them.
\end{abstract}

\maketitle

\section{Introduction}\label{S1}

Let $\mathcal K$ be an integral operator of the form
\begin{equation*}
\mathcal{K} u(x)=
\int_{\R^n}  K(x,y,x-y)u(y)
\, dy\quad (x\in\R^n)
\end{equation*}
with the kernel $K(x,y,z)$ on $\R^n_x\times\R^n_y\times\R^n_z$.
When $K(x,y,z)=K_0(z)$, it is just the
operator of convolution with $K_0$.
But if we localise it by multiplication by a function $\chi$
then we immediately have such a generalised form
with $K(x,y,z)=\chi(x)\chi(y)K_0(z)$.
A more interesting example is the case when
$K(x,y,z)$ is an oscillatory integral,
and then $\mathcal K$ is called a Fourier integral operator.

Our primary objective is to give a criterion for the global boundedness
\[
\mathcal K:L^p(\R^n) \to L^p(\R^n),\quad1<p<\infty,
\]
when we only know the local boundedness
\[
\mathcal K:L^p_{comp}(\R^n) \to L^p_{loc}(\R^n),\quad1<p<\infty,
\]
or its endpoint boundedness
\[
\mathcal K:H^1_{comp}(\R^n) \to L^1_{loc}(\R^n).
\]
Here and everywhere $H^1=H^1(\R^n)$ denotes the Hardy space
introduced by Fefferman and Stein \cite{FS}.
Let us collectively call such discussion a
{\it local-to-global boundedness argument}.
For example, if we further know the global $L^2$-boundedness and that
the endpoint local boundedness is uniform with respect to the translation
of localised regions,
then our main result (Theorem \ref{Th:main}) states that
we also have the global $L^p$-boundedness 
if the kernel $K(x,y,z)$ is finely controlled on a set away from
its singular support.
We will give the precise statement and its proof in Section \ref{S2}.

\medskip

As an important application of such local-to-global boundedness argument,
we will discuss the global $L^p$-boundedness
of the Fourier integral operators
\begin{equation}\label{FIOg}
\mathcal P u(x)
=\int_{\R^n}\int_{\R^n} e^{i\phi(x,y,\xi)}
a(x,y,\xi) u(y)\, dy  d\xi\quad(x\in \R^n).
\end{equation}
Here $\phi(x,y,\xi)$ is a real-valued function that is called a {\it phase function}
while $a(x,y,\xi)$ is called an {\it amplitude function}.
Following the theory of Fourier integral operators by H\"ormander \cite{Ho}, 
we originally assume that $\phi(x,y,\xi)$ is positively homogeneous
of order $1$ and smooth at $\xi\neq0$,
and that $a(x,y,\xi)$ is smooth and satisfies
a growth condition in $\xi$ with some $\kappa\in\R$:
\[
\sup_{(x,y)\in B}\abs{\partial_x^\alpha\partial_y^\beta\partial_\xi^\gamma
 a(x,y,\xi)}
\leq C_{\alpha\beta\gamma}^B\jp{\xi}^{\kappa-|\gamma|}\quad
(\forall \alpha, \beta, \gamma)\,;\quad 
\langle \xi\rangle=\p{1+|\xi|^2}^{1/2}
\]
for any compact set $B\subset\R^n\times\R^n$.
Then the operator $\mathcal P$ is just a microlocal
expression of the corresponding Lagrangian manifold, and with
the local graph condition, it is microlocally equivalent to the special form
\begin{equation}\label{FIOs}
P u(x)=
\int_{\R^n}\int_{\R^n} e^{i(x\cdot\xi-\varphi(y,\xi))}a(x,y,\xi)u(y)\, dy d\xi
\end{equation}
by an appropriate microlocal change of variables, with possibly another amplitude.

The local $L^p$ mapping properties of Fourier integral operators
have been extensively studied, and can be generally summarised as follows:
\begin{itemize}
\item $\mathcal P$ is $L^2_{comp}$-$L^2_{loc}$-bounded
when $\kappa\leq 0$ (H\"ormander \cite{Ho}, Eskin \cite{Es});
\item $\mathcal P$ is $L^p_{comp}$-$L^p_{loc}$-bounded
when $\kappa\leq -(n-1)|1/p-1/2|$, $1<p<\infty$ (Seeger, Sogge and Stein \cite{SSS});
\item $\mathcal P$ is $H^1_{comp}$-$L^1_{loc}$-bounded
when $\kappa\leq -\frac{n-1}{2}$ (Seeger, Sogge and Stein \cite{SSS});
\item $\mathcal P$ is locally weak $(1,1)$ type when $\kappa\leq -\frac{n-1}{2}$ 
(Tao \cite{Tao:weak11}).
\end{itemize}

The sharpness of the order $-(n-1)|1/p-1/2|$ was shown by Miyachi \cite{Mi}
and Peral \cite{Pe}
(see also \cite{SSS}). 
Therefore, the question also addressed in this paper is when Fourier integral operators are globally
$L^p$-bounded.
Although the operator $\mathcal P$ or $P$
is just a microlocal expression of
the corresponding Lagrangian manifold due to the Maslov cohomology class
(see e.g. Duistermaat \cite{Duistermaat:FIO-book-1996}), 
we still regard it as a globally defined operator
since it is still important for the applications to the theory of
partial differential equations.
Indeed, the operator $P$ is used to:
\begin{itemize}
\item
express solutions to Cauchy problems of hyperbolic equations;
\item
transform operators/equations to other simpler ones (Egorov's theorem).
\end{itemize}
The typical two types of phase functions for each analysis above are
\[
\text{
(I)\hspace{5mm} $\varphi(y,\xi)=y\cdot\xi+|\psi(\xi)|$,
\hspace{10mm}
(II)\hspace{5mm}  $\varphi(y,\xi)=y\cdot\psi(\xi)$,}
\]
where $\psi(\xi)$ is a real vector-valued smooth function which is
positively homogeneous of order $1$ for large $\xi$.
(See Definition \ref{Def:hom} for the precise meaning of this terminology).

As for the global $L^2$-boundedness of Fourier integral operators,
the following result by
Asada and Fujiwara \cite{AF} is fundamental:
\begin{thm}[\cite{AF}]\label{Th:AF}
Let $\phi(x,y,\xi)$ and $a(x,y,\xi)$ be $C^\infty$-functions, and let
\[
D(\phi):=
\begin{pmatrix}
\partial_x\partial_y\phi&\partial_x\partial_\xi\phi
\\
\partial_\xi\partial_y\phi&\partial_\xi\partial_\xi\phi
\end{pmatrix}.
\]
Assume that
$|\det D(\phi)|\geq C>0$.
Also assume that every entry of the matrix $D(\phi)$, $a(x,y,\xi)$
and all their derivatives are bounded. 
Then operator $\mathcal P$ defined by \eqref{FIOg} is $L^2(\R^n)$-bounded.
\end{thm}
The result of \cite{AF} was used to construct
the solution to the
Cauchy problem of Schr\"odinger equations by means of
the Feynman path integrals in Fujiwara \cite{Fu}.
For the operator $P$ defined by \eqref{FIOs},
the conditions of Theorem \ref{Th:AF} are reduced to
a global version of the local graph condition
\begin{equation}\label{graphcond}
\abs{\det \partial_y\partial_\xi\varphi(y,\xi)}\geq C>0,
\end{equation}
and the growth conditions
\begin{equation}\label{growth}
\begin{aligned}
&\abs{\partial_y^\alpha\partial_\xi^\beta\varphi(y,\xi)} \le
C_{\alpha\beta}\quad
(\forall\, |\alpha+\beta|\geq2, |\beta|\geq1),
\\
&\abs{\partial_x^\alpha\partial_y^\beta\partial_\xi^\gamma a(x,y,\xi)}
\le C_{\alpha\beta\gamma}\quad
(\forall \alpha, \beta, \gamma),
\end{aligned}
\end{equation}
for all $x,y,\xi\in\R^n$. 
Note that the local graph condition 
is required even for the local $L^2$-boundedness
of Fourier integral operators of order zero, so it is rather natural to assume
\eqref{graphcond} for the global $L^2$-boundedness.
We also note that the phase functions of the type (I)
satisfy the growth condition \eqref {growth}, but the type (II) does not.
We mention that other types of growth conditions
were introduced by the authors in \cite{RS}
to obtain the global $L^2$-boundedness for operators with phase functions
of the type (II), and such result was then used to show global smoothing
estimates for dispersive equations in a series of papers \cite{RS2},
\cite{RS3} and \cite{RS4}.

As for the global $L^p$-boundedness, it is deduced by
our local-to-global argument
from the global $L^2$-boundedness (Theorem \ref{Th:AF}) and
the local endpoint result given by Seeger, Sogge and Stein \cite{SSS}.
Indeed, in this paper we establish the following generalised result:
\begin{thm}\label{main}
Let $\varphi(y,\xi)$ and $a(x,y,\xi)$ be $C^\infty$-functions.
Assume that
 $\varphi(y,\xi)$ is positively homogeneous of order $1$ for large $\xi$
and satisfies \eqref{graphcond}.
Also assume that
\begin{align*}
&\abs{\partial_y^\alpha\partial_\xi^\beta(y\cdot\xi-\varphi(y,\xi))}
\le
C_{\alpha\beta}\jp{\xi}^{1-|\beta|}\quad
(\forall\, \alpha, |\beta|\geq1),
\\
&\abs{\partial_x^\alpha\partial_y^\beta\partial_\xi^\gamma a(x,y,\xi)}
\leq C_{\alpha\beta\gamma}
\jp{\xi}^{ -(n-1)|1/p-1/2|-|\gamma|}\quad
(\forall \alpha, \beta, \gamma),
\end{align*}
hold for all $x,y,\xi\in\R^n$.
Then operator $P$ defined by \eqref{FIOs} is $L^p(\R^n)$-bounded, for every $1<p<\infty$.
\end{thm}
Theorem \ref{main} together with some related results
will be restated in Section \ref{S3} in a different form
(in particular, Theorem \ref{main} follows from Corollary \ref{Cor1}),
emphasising that they are given as an application of
our local-to-global argument
discussed in Section \ref{S2}.
We remark that Theorem \ref{main} with $p=2$ was also given by
Kumano-go \cite{Ku}.
For the special cases $\varphi(y,\xi)=\varphi(\xi)$ and $a(x,y,\xi)=a(\xi)$,
Theorem \ref{main} was given by Miyachi \cite{Mi}
under the assumptions that $\varphi>0$ and that the compact hypersurface
\[
\Sigma=\b{\xi\in\R^n\setminus0\,:\,\varphi(\xi)=1}
\]
has non-zero Gaussian curvature.
Beals \cite{Be} and Sugimoto \cite{Su} discussed the case when
$\Sigma$ might have vanishing Gaussian curvature but is still convex.



We also mention that phase functions of the type (I) again satisfy
the assumption of Theorem \ref{main}, but type (II) does not.
Unfortunately our local-to-global argument does not work for the type (II)
and the linear growth in $y$ causes an extra requirement for the
growth order of amplitude functions. This is rather natural since one knows that in general, in the type (II) case, there is a loss in weight in global $L^p$ estimates for $p\not=2$, 
see Coriasco and Ruzhansky \cite{CR-CRAS, CR} for global bounds for Fourier integral operators in this case.

To complement some references on the local and global boundedness properties of Fourier integral operators,
we refer to 
the authors' paper \cite{RS-weighted:MN} for the weighted $L^2$- and to
Dos Santos Ferreira and Staubach \cite{Staubach-Dos-Santos-Ferreira:local-global-FIOs-MAMS}
for other weighted properties of Fourier integral operators, to
Rodr\'iguez-L\'opez and Staubach \cite{Rodriguez-Lopez-Staubach:rough-FIOs} for estimates for
rough Fourier integral operators, to \cite{Ruzhansky:CWI-book} for
$L^p$-estimates for Fourier integral operators with complex phase functions, as well as to
\cite{Ruzhansky:FIOs-local-global} for an earlier overview of local and global properties of
Fourier integral operators with real and complex phase functions.
The $L^{p}$-boundedness of bilinear Fourier integral operators has been also investigated,
see e.g. Hong, Lu, Zhang \cite{GLU} and references therein.

%


\section{A Local-to-global boundedness argument}\label{S2}

We discuss when the local boundedness of an integral operator
induces the global one.
Let $\mathcal K$ be an integral operator of the form
\begin{equation}\label{kernelrep}
\mathcal{K} u(x)=
\int_{\R^n}  K(x,y,x-y)u(y)
\, dy\quad (x\in\R^n)
\end{equation}
with a measurable function $K(x,y,z)$ on $\R^n_x\times\R^n_y\times\R^n_z$.
The formal adjoint $\mathcal K^*$ of $\mathcal K$ is given by
\begin{equation}\label{kernelrep_*}
\mathcal{K}^* u(x)=
\int_{\R^n}  K^*(x,y,x-y)u(y)\, dy,\quad  K^*(x,y,z)=\overline{K(y,x,-z)}.
\end{equation}
We introduce a notion of the local boundedness.
By $\chi_B$ we denote the multiplication
by the smooth characteristic function of the set $B\subset\R^n$.
As before, $H^1(\R^n)$ denotes the Hardy space introduced by Fefferman-Stein \cite{FS}.

\begin{defn}
We say that the operator $\mathcal K$
is $H^1_{comp}(\R^n)$-$L^1_{loc}(\R^n)$-bounded if the localised operator
$\chi_{B} \mathcal K \chi_{B}$ is $H^1(\R^n)$-$L^1(\R^n)$-bounded
for any compact set $B\subset\R^n$.
Furthermore, if the operator norm of $\chi_{B_h} \mathcal K \chi_{B_h}$
is  bounded in $h\in\R^n$ for the translated set $B_h=\{x+h:x\in B\}$
of any compact set $B\subset\R^n$, i.e. if
$$
\sup_{h\in\R^n} \|\chi_{B_h} \mathcal K \chi_{B_h}\|_{H^1(\R^n)\to L^1(\R^n)}<\infty,
$$
we say that
the operator $\mathcal K$
is {\em uniformly} $H^1_{comp}(\R^n)$-$L^1_{loc}(\R^n)$-bounded.
\end{defn}
If we introduce the translation operator $\tau_h:f(x)\mapsto f(x-h)$ and
its inverse (formal adjoint) $\tau_h^*=\tau_{-h}$,
we have the equality $\chi_{B_h}=\tau_h\chi_{B}\tau_h^*$.
Since $L^1$ and $H^1$ norms are translation invariant,
$\mathcal K$
is uniformly $H^1_{comp}$-$L^1_{loc}$-bounded
if and only if $\chi_{B}(\tau_h^* \mathcal K \tau_h)\chi_{B}$ is
$H^1$-$L^1$-bounded
for any compact set
$B\subset\R^n$ and the operator norms are bounded in $h\in\R^n$.
We remark that the operator $\tau_h^* \mathcal K \tau_h$
has the expression
\begin{equation}\label{kernelrep_h}
\tau_h^* \mathcal K \tau_h u(x)=
\int_{\R^n}  K_h(x,y,x-y)u(y)
\, dy,\quad
K_h(x,y,z)=K(x+h,y+h,z)
\end{equation}

We have the following main result:
\begin{thm}\label{Th:main}
Suppose that operator $\mathcal K$ defined by \eqref{kernelrep}
is $L^2(\R^n)$-bounded and uniformly
$H^1_{comp}(\R^n)$-$L^1_{loc}(\R^n)$-bounded.
Assume that there exits a measurable function $H(x,y,z)$ which satisfies the
following condition:
\begin{itemize}
\item[(A1)]
There exist constants $d>0$ and $k>n$ such that
\[
\sup_{H(x,y,z)\geq d}
\abs{H(x,y,z)^kK(x,y,z)}<\infty.
\]
\end{itemize}
Furthermore, we set
\[
\widetilde{H}(z):=\inf_{x,y\in\R^n}H(x,y,z).
\]
and assume also the following two conditions:
\begin{itemize}
\item[(A2)]
There exist constants $A>0$ and $A_0>0$ such that
$$\text{$\widetilde H (z)\geq A_0 |z|$ whenever $|z|\geq A$.}$$
\item[(A3)]
There exist constants $b>0$ and $b_0>0$ such that
$$\text{$\widetilde H (z)\leq b_0\widetilde H(z-z')$
whenever $\widetilde H(z)\geq b|z'|$.}$$
\end{itemize}
Then $\mathcal K$ is $H^1(\R^n)$-$L^1(\R^n)$-bounded.
If in addition operator $\mathcal K^*$ defined by \eqref{kernelrep_*}
is uniformly $H^1_{comp}(\R^n)$-$L^1_{loc}(\R^n)$-bounded, then
$\mathcal K^*$ is also $H^1(\R^n)$-$L^1(\R^n)$-bounded.
\end{thm}

Theorem \ref{Th:main} means that
the global $L^2$-boundedness and some additional assumptions
induce the global $H^1$-$L^1$-boundedness form the local one.
Then, if we want, we can have the global $L^p$-boundedness 
for $1<p<\infty$ by the
interpolation and the duality argument. 
Immediate examples to which
Theorem \ref{Th:main} can be applied
are pseudo-differential operators
\begin{align*}
\mathcal P_s u(x)
&=\iint_{\R^n\times\R^n} e^{i(x-y)\cdot\xi}a(x,y,\xi)u(y)\, dy d\xi
\\
&=\int_{\R^n} K(x,y,x-y)u(y)\, dy,
\end{align*}
where
\[
K(x,y,z)
=\int_{\R^n}e^{iz\cdot\xi}a(x,y,\xi)\,d\xi.
\]
(Indeed, we can take $H(x,y,z)=|z|$ in this case if $a(x,y,\xi)$
belongs to a standard symbol class.)
More interesting examples for which Theorem \ref{Th:main} yields
new $L^p$ boundedness results are Fourier integral operators which include
pseudo-differential operators as special ones.
They will be intensively discussed in the next section.

Now we give the proof of Theorem \ref{Th:main}.
We only show the assertion for $\mathcal K$ because
conditions (A1)--(A3) induce corresponding conditions
for the kernel $ K^*(x,y,z)$ in \eqref{kernelrep_*} of $\mathcal K^*$ 
 if we take $H^*(x,y,z)=H(y,x,-z)$.
Furthermore, we may take $d=1$ in (A1) otherwise
replace $H$ by $H/d$, and we may also take $b=b_0$ in (A3) otherwise
replace the smaller one by the bigger.

We introduce the notations
\[
\Delta_r:=\{(x,y,z)\in\R^n\times\R^n\times\R^n:H(x,y,z)\geq r\}
\]
and
\[
\widetilde\Delta_r:=\{z\in\R^n:\widetilde H(z)\geq r\}.
\]
Clearly we have the monotonicity of $\Delta_r$ and $\widetilde\Delta_r$
in $r>0$, that is,
$\Delta_{r_1}\subset\Delta_{r_2}$,
$\widetilde\Delta_{r_1}\subset\widetilde\Delta_{r_2}$
for $r_1\geq r_2\geq0$.
On account of them, we have the following:
\begin{lem}\label{Lem:L1est}
Let $r\geq1$ and let $h\in\R^n$.
Suppose $\supp f\subset \{x\in \R^n:|x|\leq r\}$.
Then we have
\[
\left\|\tau_h^*\mathcal{K}\tau_h f\right\|_{L^1(\widetilde\Delta_{br})}
\leq C\,\n{f}_{L^1},
\]
where $C$ is a positive constant independent of $r$ and $h$.
\end{lem}
\begin{proof}
First we consider the case $h=0$.
For $x\in\widetilde\Delta_{br}$ and $|y|\leq r$, we have
$\widetilde H(x)\geq br$ by the definition of $\Delta_{br}$,
and hence we also have
$\widetilde H(x)\geq b|y|$.
Then from (A3) with $b=b_0$, we obtain
$$br\leq\widetilde H(x)\leq b\widetilde H(x-y)\leq bH(x,y,x-y)$$
which implies $(x,y,x-y)\in\Delta_{r}$ and
$\widetilde H(x,y,x-y)^{-1}\leq b \widetilde H(x)^{-1}$.
Then we have
\begin{align*}
|\mathcal{K} f(x)|
&\leq b^k\widetilde H(x)^{-k}\int_{|y|\leq r}
\abs{H(x,y,x-y)^kK(x,y,x-y)f(y)}\,dy
\\
&\leq b^k\widetilde H(x)^{-k}
\n{H(x,y,z)^kK(x,y,z)}_{L^\infty(\Delta_r)}
\n{f}_{L^1},
\end{align*}
for $x\in\widetilde\Delta_{br}$.
Hence, 
by the monotonicity $\Delta_{br}\subset\Delta_{b}$
and $\widetilde\Delta_{r}\subset\widetilde\Delta_{1}$ ($r\geq1$), we have
\begin{align*}
\left\|\mathcal{K} f\right\|_{L^1(\widetilde\Delta_{br})}
&\leq b^k
\left\|\widetilde H(x)^{-k}\right\|_{L^1(\widetilde\Delta_{br})}
\n{H(x,y,z)^kK(x,y,z)}_{L^\infty(\Delta_r)}\n{f}_{L^1}
\\
&\leq b^k\left\|\widetilde H(x)^{-k}
\right\|_{L^1(\widetilde\Delta_{b})}
\n{H(x,y,z)^kK(x,y,z)}_{L^\infty(\Delta_1)}\n{f}_{L^1}
\\
&\leq C\n{f}_{L^1},
\end{align*}
for $k>n$,
where we have used (A2) to justify the estimate
\begin{align*}
\n{\widetilde H(z)^{-k}}_{L^1(\widetilde\Delta_b)}
&\leq
\n{\widetilde H(z)^{-k}}_{L^1(\widetilde\Delta_b\cap\b{|z|\leq A})}+
\n{\widetilde H(z)^{-k}}_{L^1(\widetilde\Delta_b\cap\b{|z|\geq A})}
\\
&\leq
b^{-k}\n{1}_{L^1(|z|\leq A)}+
A_0^{-k}\n{|z|^{-k}}_{L^1(|z|\geq A)}
\\
&\leq C,
\end{align*}
and also (A1) with $d=1$.

For general $h\in\R^n$, we apply the same argument for $K_h$
in \eqref {kernelrep_h} and
\[
H_h(x,y,z)=H(x-h,y-h,z)
\]
instead of $K$ and $H$, respectively.
We remark that conditions (A1), (A2) and (A3) in Theorem \ref{Th:main} are
invariant in $h\in\R^n$ in the sense that we have
\begin{align*}
&\sup_{H_h(x,y,z)\geq d}
\abs{H_h(x,y,z)^kK_h(x,y,z)}
=
\sup_{H(x,y,z)\geq d}
\abs{H(x,y,z)^kK(x,y,z)},
\\
&\widetilde{H_h}(z)=\inf_{x,y\in\R^n}H_h(x,y,z)=\widetilde{H}(z).
\end{align*}Then we have the same estimates with the same constants
but $\mathcal K$ replaced by
$\tau_h^*\mathcal{K}\tau_h$.
This finishes the proof.
\end{proof}

\begin{lem}\label{Lem:outside}
Let $r\geq1$.
Then there exists a constant $c>0$ independent of $r$ such that
$\R^n\setminus\widetilde\Delta_{br}\subset\b{z:|z|< c\,r}$.
\end{lem}
\begin{proof}
Let $z\in\R^n\setminus\widetilde\Delta_{br}$, that is, $\widetilde H(z)<br$.
By the definition of $\widetilde H(z)$,
there exist $x_0,y_0,\xi_0\in\R^n$ such that
$H(x_0,y_0,z)\leq 2\widetilde H(z)$, and by (A2) we have
$A_0|z|\leq H(x_0,y_0,z)$ for $|z|\geq A$.
Hence we have $|z|\leq (2b/A_0)\,r$ for $|z|\geq A$.
On the other hand, we always have $|z|\leq Ar$ for $|z|\leq A$
since $r\geq1$,
and we have the conclusion.
\end{proof}
Now we are ready to prove the $H^1$-$L^1$-boundedness.
We use the characterisation of $H^1$ by the atomic
decomposition proved by Coifman and Weiss \cite{CW}.
That is, any $f\in H^1(\R^n)$ can be
represented as 
$$
f=\sum_{j=1}^\infty\lambda_jg_j,\quad\lambda_j\in\C,\quad g_j:\textrm{ atom},
$$
and the norm $\|f\|_{H^1}$ is equivalent to the norm
$\n{\b{\lambda_j}_{j=1}^\infty}_{\ell^1}=\sum^\infty_{j=1}|\lambda_j|$.
Here we call a function $g$
on $\R^n$ an atom if there is a ball $B=B_g\subset\R^n$ such that
$\supp g\subset B$, $\|g\|_{L^\infty}\leq|B|^{-1}$ ($|B|$ is the
Lebesgue measure of the ball $B$) and
$\int g(x)\,dx=0$.
From this, all we
have to show is the estimate
$$
\left\|\mathcal{K}g\right\|_{L^1(\R^n)}\leq C,
$$
with some constant $C>0$ for all atoms $g$.
By an appropriate translation,
it is further reduced to the estimate
$$
\left\|\tau_h^*\mathcal{K}\tau_hf\right\|_{L^1(\R^n)}\leq C,
\quad f\in \A_r,
$$
where $\A_r$
is the set of all functions $f$ on $\R^n$
such that
\begin{equation}\label{EQ:atoms}
\supp f\subset B_r=\{x\in\R^n:|x|\leq r\},\quad\|f\|_{L^\infty}\leq |B_r|^{-1},
\quad\int f(x)\,dx=0.
\end{equation} 
Here and hereafter in this section, $C$ always denotes
a constant which is independent of $h\in\R^n$ and
$0<r<\infty$, and which may differ from one formula to another.

Suppose $f\in\A_{r}$ with $r\geq 1$.
Then we split $\R^n$ into two parts
$\widetilde\Delta_{br}$ and $\R^n\setminus\widetilde\Delta_{br}$.
For the part $\widetilde\Delta_{br}$,
we have by Lemma \ref{Lem:L1est} that
\[
\left\|\tau_h^*\mathcal{K} \tau_hf\right\|_{L^1(\widetilde\Delta_{br})}
\leq C\n{f}_{L^1}\leq C.
\]
For the part $\R^n\setminus\widetilde\Delta_{br}$,
we have
by Lemma \ref{Lem:outside} and the Cauchy-Schwarz inequality
\begin{align*}
\left\|\tau_h^*\mathcal{K} \tau_hf\right\|_{L^1(\R^n\setminus\widetilde\Delta_{br})}
&\leq
\|1\|_{L^2(|x|< c\,r)}
\left\|\tau_h^*\mathcal{K} \tau_hf\right\|_{L^2(\R^n)}
\\
&\leq
Cr^{n/2}\|f\|_{L^2(\R^n)}
\leq C,
\end{align*}
where we have used the assumption that
$\mathcal{K}$ is $L^2$-bounded, and \eqref{EQ:atoms} in the last inequality.

Suppose now $f\in\A_{r}$ with $r\leq 1$.
Then we split $\R^n$ into the
parts $\Delta_{b}$ and $\R^n\setminus\Delta_b$.
For the part  $\Delta_{b}$,
we have by Lemma \ref{Lem:L1est} with $r=1$ and the inclusion
$\supp f\subset B_r\subset B_1$
\[
\left\|\tau_h^*\mathcal{K} \tau_hf\right\|_{L^1(\widetilde\Delta_{b})}
\leq C\n{f}_{L^1}\leq C.
\]
For the part $\R^n\setminus\Delta_b$,
we have by Lemma \ref{Lem:outside} that
\begin{align*}
\left\|\tau_h^*\mathcal{K} \tau_h f\right\|_{L^1(\R^n\setminus\Delta_{b})}
&\leq
\left\|\tau_h^*\mathcal{K} \tau_h f\right\|_{L^1(|x|<c)}
\\
&\leq C\n{f}_{H^1}
\leq C,
\end{align*}
where we have used the fact that $\mathcal{K}$ is uniformly
$H^1_{comp}$-$L^1_{loc}$-bounded.
Thus the proof of Theorem \ref{Th:main}
is complete.


\section{Fourier integral Operators}\label{S3}
A typical example of integral operators \eqref{kernelrep}
which we have in mind is Fourier integral operators of the form
\begin{equation}\label{FIOp}
\mathcal{P} u(x)=
\int_{\R^n}\int_{\R^n} e^{i\phi(x,y,\xi)}a(x,y,\xi)u(y)
\, dy d\xi\quad (x\in\R^n).
\end{equation}
For convenience, we introduce the function $\Phi(x,y,\xi)$
to write
\begin{equation}\label{Phase}
\phi(x,y,\xi)=(x-y)\cdot\xi+\Phi(x,y,\xi),
\end{equation}
and then we have the kernel representation
\[
\mathcal{P} u(x)=
\int_{\R^n}  K(x,y,x-y)u(y)
\, dy
\]
with
\begin{equation}\label{kernel}
K(x,y,z)
=\int_{\R^n}e^{i\{z\cdot\xi+\Phi(x,y,\xi)\}}a(x,y,\xi)\,d\xi.
\end{equation}
In particular, $\mathcal{P}$ is a pseudo-differential operator
when $\Phi(x,y,\xi)=0$.
We remark that the formal adjoint $\mathcal P^*$ of $\mathcal P$ is
of the same form \eqref{FIOp} with the replacement
\begin{equation}\label{adjoint}
\begin{aligned}
&\Phi(x,y,\xi)\longmapsto\Phi^*(x,y,\xi)=-\Phi(y,x,\xi),
\\
&a(x,y,\xi)\longmapsto a^*(x,y,\xi)=\overline{a(y,x,\xi)},
\end{aligned}
\end{equation}
and also the operator $\tau_h^*\mathcal{P}\tau_h$
with
\begin{equation}\label{translation}
\begin{aligned}
&\Phi(x,y,\xi)\longmapsto\Phi^h(x,y,\xi)=\Phi(x+h,y+h,\xi),
\\
&a(x,y,\xi)\longmapsto a^h(x,y,\xi)=a(x+h,y+h,\xi),
\end{aligned}
\end{equation}
as special cases of the general rules 
\eqref{kernelrep_*} and \eqref{kernelrep_h}.

We introduce a class of amplidtude functions $a(x,y,\xi)$:
\begin{defn}\label{symbol}
For $\kappa\in\R$, $S^\kappa$ denotes
the class of smooth functions
$a=a(x,y,\xi)\in C^\infty(\R^n\times\R^n\times\R^n)$
satisfying the estimate
\[
\abs{
\partial^\alpha_x \partial^\beta_y \partial^\gamma_\xi a(x,y,\xi)
}
\leq C_{\alpha\beta\gamma}\jp{\xi}^{\kappa-|\gamma|}
\]
for all $x,y,\xi\in\R^n$ and all multi-indices $\alpha,\beta,\gamma$.
\end{defn}

Let us now try to apply Theorem \ref{Th:main} from the previous section for
Fourier integral operators defined by \eqref{FIOp}.
Our natural choice of $H(x,y,z)$ is a defining function of the singular
support of the kernel.
For example, the kernels of pseudo-differential operators 
(that is, \eqref{kernel} with $\Phi(x,y,z)=0$) is singular
only when $z=0$, and we can take $H(x,y,z)=|z|$.
Indeed we can easily see that it satisfies assumptions (A1)--(A3)
in Theorem \ref{Th:main} if $a(x,y,\xi)$ belongs to a class $S^m$.
For general Fourier integral operators, we can find $H(x,y,z)$ 
corresponding to $\Phi(x,y,z)$ by the same consideration:
\begin{lem}\label{Lem:main}
Assume that $a=a(x,y,\xi)\in S^{\kappa}$ with some $\kappa\in\R$, and assume
also that $\Phi(x,y,\xi)$ is a real-valued $C^\infty$-function
and that $\partial^\gamma_\xi\Phi(x,y,\xi)\in S^0$
for $|\gamma|=1$.
Let
\[
H(x,y,z):=\inf_{\xi\in\R^n}\abs{z+\nabla_\xi\Phi(x,y,\xi)},
\]
and let
\[
\widetilde H(z):=\inf_{x,y\in\R^n}H(x,y,z)=
\inf_{x,y,\xi\in\R^n}\abs{z+\nabla_\xi\Phi(x,y,\xi)}.
\]
Then $K(x,y,z)$ defined by \eqref{kernel} satisfies 
assumptions {\rm (A1)--(A3)} in Theorem \ref{Th:main}.
\end{lem}
\begin{proof}
The expression \eqref{kernel} is justified by the
integration by parts
$$
K(x,y,z)=\int_{\R^n}e^{i\{z\cdot\xi+\Phi(x,y,\xi)\}}
\left(L^*\right)^{n+1} a(x,y,\xi)\,d\xi
$$
outside the set
\begin{align*}
\Sigma&=
\{(x,y,-\nabla_\xi\Phi(x,y,\xi))\in\R^n\times\R^n\times\R^n:x,y,\xi\in\R^n\}
\\
&=\{(x,y,z)\in\R^n\times\R^n\times\R^n:H(x,y,z)=0 \},
\end{align*}
where $L^*$ is the transpose of the operator
$$
L=\frac{(z+\nabla_\xi\Phi)\cdot\nabla_\xi}{i|z+\nabla_\xi\Phi|^2}.
$$
Noticing that $d\leq H(x,y,z)$ implies $d\leq|z+\nabla_\xi\Phi(x,y,\xi)|$
for all $\xi\in\R^n$, we easily have (A1).
On the other hand,
we have
$$|z|\leq |z+\nabla_\xi\Phi(x,y,\xi)|+|\nabla_\xi\Phi(x,y,\xi)|
\leq |z+\nabla_\xi\Phi(x,y,\xi)|+N,$$
with some constant $N>0$
for any $x,y,\xi\in\R^n$,
hence $|z|\leq \widetilde H(z)+N$.
Then for $|z|\geq 2N$ we have $|z|\leq \widetilde H(z)+|z|/2$,
hence $\widetilde H(z)\geq|z|/2$, that is, (A2).
Finally, if $\widetilde H(z)\geq2|z'|$,
then we have
\begin{align*}
\widetilde H(z)
&\leq H(x,y,z)\leq|z+\nabla_\xi\Phi(x,y,\xi)|
\leq|z-z'+\nabla_\xi\Phi(x,y,\xi)|+|z'| 
\\
&\leq|z-z'+\nabla_\xi\Phi(x,y,\xi)|+\widetilde H(z)/2
\end{align*}
hence $\widetilde H(z)\leq2|z-z'+\nabla_\xi\Phi(x,y,\xi)|$
for all $x,y,\xi\in\R^n$, hence  $\widetilde H(z)\leq2\widetilde H(z-z')$
that is, we have (A3).
\end{proof}

From Lemma \ref{Lem:main}, we immediately obtain
the following result from Theorem \ref{Th:main}:
\begin{thm}\label{Th:mainFIO}
Let $\mathcal P$ be a operator defined by \eqref{FIOp} with \eqref{Phase}.
Let $1<p<\infty$, let $\kappa_1\leq0$, and let $\kappa\leq2\kappa_1|1/p-1/2|$
Assume the following conditions:
\begin{itemize}
\item[(B1)]
$\Phi(x,y,\xi)$ is a real-valued $C^\infty$-function
and $\partial^\gamma_\xi\Phi(x,y,\xi)\in S^0$
for $|\gamma|=1$.
\item[(B2)]
$\mathcal P$ is
$L^2(\R^n)$-bounded whenever $a(x,y,\xi)\in S^{0}$.
\item[(B3)]
$\mathcal P$ and $\mathcal P^*$ are uniformly
$H^1_{comp}(\R^n)$-$L^1_{loc}(\R^n)$-bounded
whenever $a=a(x,y,\xi)\in S^{\kappa_1}$.
\end{itemize}
Then $\mathcal P$ is $L^p(\R^n)$-bounded
for any $a=a(x,y,\xi)\in S^{\kappa}$.
\end{thm}
Indeed, by Theorem \ref{Th:main} and Lemma \ref{Lem:main},
assumptions (B1)--(B3) induce the $H^1$-$L^1$-boundedness 
of the operators $\mathcal P$ and  $\mathcal P^*$
for $a=a(x,y,\xi)\in S^{\kappa_1}$
if we notice that $S^{\kappa_1}\subset S^0$.
Then by the duality and the complex interpolation argument,
we have the $L^p$-boundedness of $\mathcal P$ with the critical case
$a=a(x,y,\xi)\in S^{2\kappa_1|1/p-1/2|}$, hence also for $a=a(x,y,\xi)\in S^{\kappa}$
since $S^\kappa\subset S^{2\kappa_1|1/p-1/2|}$.
We remark that we can also conclude 
the $L^\infty$-$BMO$-boundedness of $\mathcal P$
if $a=a(x,y,\xi)\in S^{\kappa_1}$.

Assumptions (B2) and (B3) in Theorem \ref{Th:mainFIO}
are essentially the requirements
for phase functions $\Phi(x,y,\xi)$.
A condition for (B2) is given by Asada and Fujiwara \cite{AF},
while (B3) is given by Seeger, Sogge and Stein \cite{SSS}.
We state further conclusions based on them by
restricting our phase functions to the form
\[
\phi(x,y,\xi)=x\cdot\xi-\varphi(y,\xi)
\quad(\text{in other words
$\Phi(x,y,\xi)=y\cdot\xi-\varphi(y,\xi)$}).
\]
First we make precise the notion of homogeneity:
\begin{defn}\label{Def:hom}
We say that
{\it $\varphi=\varphi(y,\xi)$ is positively homogeneous of order $1$} if
\begin{equation}\label{EQ:hom}
\varphi(y,\lambda\xi)=\lambda\varphi(y,\xi)
\end{equation}
holds for all $y\in\R^n$,
$\xi\neq0$ and $\lambda>0$.
We also say that 
{\it $\varphi=\varphi(y,\xi)$ is positively homogeneous of order $1$
for large $\xi$}
if there exist a constant $R>0$ such that \eqref{EQ:hom}
holds for all $y\in\R^n$,
$|\xi|\geq R$ and $\lambda\geq1$.
\end{defn}
For the operator of the form
\begin{equation}\label{FIOp1}
P u(x)=
\int_{\R^n}\int_{\R^n} e^{i(x\cdot\xi-\varphi(y,\xi))}a(x,y,\xi)u(y)\, dy d\xi
\quad (x\in\R^n),
\end{equation}
we have the following boundedness:
\begin{cor}\label{Cor1}
Let $1<p<\infty$ and let $\kappa\leq-(n-1)|1/p-1/2|$.
Assume that $a=a(x,y,\xi)\in S^\kappa$.
Assume also the following conditions:
\begin{itemize}
\item[(C1)]
$\varphi(y,\xi)$ is a real-valued $C^\infty$-function
and $\partial^\gamma_\xi(y\cdot\xi-\varphi(y,\xi))\in S^0$
for $|\gamma|=1$.
\item[(C2)]
There exists a constant $C>0$ such that
$\abs{\det \partial_y\partial_\xi\varphi(y,\xi)}\geq C$
for all $y,\xi\in\R^n$.
\item[(C3)]
$\varphi(y,\xi)$ is positively homogeneous
of order $1$ for large $\xi$.
\end{itemize}
Then operator $P$ defined by \eqref{FIOp1} is $L^p(\R^n)$-bounded.
\end{cor}

\begin{proof}
Let us induce assumptions (B1)--(B3) of Theorem \ref{Th:mainFIO}
with $\kappa_1=-(n-1)/2$
from the assumptions (C1)--(C3) of Corollary \ref{Cor1}
for the special case $\phi(x,y,\xi)=x\cdot\xi-\varphi(y,\xi)$ or,
in other words, for $\Phi(x,y,\xi)=y\cdot\xi-\varphi(y,\xi)$.
We remark that (B1) is just an interpretation of assumption (C1).
As for (B2),
a sufficient condition for the $L^2$-boundededness of $\mathcal P$
is known from Asada and Fujiwara \cite{AF}, that is, Theorem \ref{Th:AF}
in Introduction.
In particular,
(B2) is fulfilled if (C1) and (C2) are satisfied.

Let us discuss (B3).
A sufficient condition for the $H^1_{comp}$-$L^1_{loc}$-boundedness
of $P$ 
is known by the work of Seeger, Sogge and Stein \cite{SSS},
that is, $P$ is $H^1_{comp}$-$L^1_{loc}$-bounded
for $a=a(x,y,\xi)\in S^{-(n-1)/2}$
if $\varphi(y,\xi)$ is a real-valued $C^\infty$-function
on $\R^n\times(\R^n\setminus0)$ and positively homogeneous of order $1$.
If we carefully trace the argument in \cite{SSS},
we can say that
$\chi_{K} P \chi_{K}$ is $H^1(\R^n)$-$L^1(\R^n)$-bounded
for any compact set $K\subset\R^n$
and its operator norm is bounded by a constant depending only on
$n$, $K$ and quantities
\begin{align*}
&M_\ell=\sum_{|\alpha|+|\beta|+|\gamma|\leq \ell}\sup_{x,y,\xi\in\R^n}
|\partial^\alpha_x\partial^\beta_y
\partial^\gamma_\xi a(x,y,\xi)\jp{\xi}^{(n-1)/2+|\gamma|)}|,
\\
&N_\ell=\sum_{\substack{|\beta|\leq\ell \\ 1\leq|\gamma|\leq \ell}}
\sup_{\substack{x,y\in\R^n\\ \xi\neq0}}
|\partial^\beta_y
\partial^\gamma_\xi (y\cdot\xi-\varphi(y,\xi))\abs{\xi}^{-(1-|\gamma|)}|,
\end{align*}
with some large $\ell$.
The same is true for $P^*$ if we trace the argument in \cite{St2} instead
but we require (C2) in this case.
Then $P$ and $P^*$ are uniformly $H^1_{comp}$-$L^1_{loc}$-bounded
if $M_\ell$ and $N_\ell$ are finite
since the quantities $M_\ell$ and $N_\ell$ are invariant under the replacements
\eqref{adjoint} and \eqref{translation}.

Based on this fact, $P$ and $P^*$ are uniformly
$H^1_{comp}$-$L^1_{loc}$-bounded if $a\in S^{-(n-1)/2}$
under the assumptions (C1)--(C3).
In fact, let us split $a(x,y,\xi)$ into the sum of $a(x,y,\xi)g(\xi)$ and
$a(x,y,\xi)(1-g(\xi))$ with an appropriate smooth cut-off function
$g\in C^\infty_0(\R^n)$ which is equal to $1$ near the origin.
Then for the terms $P_1$ and $P_1^*$ corresponding to $a(x,y,\xi)(1-g(\xi))$,
we can regard $\varphi(y,\xi)$ as a positively homogeneous function
of order $1$ by a modification near $\xi=0$, and they are
uniformly $H^1_{comp}$-$L^1_{loc}$-bounded by the above observation.
On the other hand, the terms $P_2$ and $P_2^*$ corresponding
to $a(x,y,\xi)g(\xi)$ are 
$L^1$-bounded (hence uniformly $H^1_{comp}$-$L^1_{loc}$-bounded)
because 
\begin{align*}
&P_2u(x)=\int  K(x,y) u(y)\,dy,\quad P_2^*u(x)=\int \overline{K(y,x)} u(y)\,dy,
\\
& K(x,y)
=\int_{\R^n}e^{i(x\cdot\xi-\phi(y,\xi))}
a(x,y,\xi)g(\xi)\,d\xi,
\end{align*}
and the integral kernel $K(x,y)$
is integrable in both $x$ and $y$.
This fact can be verified by the integration by parts
\begin{align*}
K(x,y)
&=(1+|x-y|^2)^{-n} \int_{\R^n}(1-\Delta_\xi)^n e^{i(x-y)\cdot\xi}\cdot
e^{i(y\cdot\xi-\phi(y,\xi))}a(x,y,\xi)g(\xi)\,d\xi
\\
&=(1+|x-y|^2)^{-n}\int_{\R^n} e^{i(x-y)\cdot\xi}\cdot
(1-\Delta_\xi)^n \{e^{i(y\cdot\xi-\phi(y,\xi))}a(x,y,\xi)g(\xi)\}\,d\xi
\end{align*}
followed by the the conclusion
\[
\abs{K(x,y)}\leq C(1+|x-y|^2)^{-n}
\]
because of assumptions (C1), $a\in S^{-(n-1)/2}$, and $g\in C^\infty_0$.

As a conclusion, (B3) is fulfilled if (C1)--(C3) are satisfied, and
the proof of Corollary \ref{Cor1} is complete.
\end{proof}


We can admit positively homogeneous phase functions
which might have singularity
at the origin for a special kind of operators of the form
\begin{equation}\label{FIOp2}
T u(x)=\int_{\R^n} e^{i(x\cdot\xi+\psi(\xi))}a(x,\xi)\widehat{u}(\xi)
\, d\xi\quad (x\in\R^n).
\end{equation}
For such operators we have the following boundedness:
\begin{cor}\label{Cor2}
Let $1<p<\infty$ and let $\kappa\leq-(n-1)|1/p-1/2|$.
Assume that $a=a(x,\xi)\in S^\kappa$ and
that $\psi=\psi(\xi)$ is a real-valued $C^\infty$-function on $\R^n\setminus0$
which is positively homogeneous
of order $1$.
Then the operator $T$ defined by \eqref{FIOp2} is $L^p(\R^n)$-bounded.
\end{cor}

\begin{proof}
Again we spilt the amplitude $a(x,\xi)$ into the sum of $a(x,\xi)g(\xi)$ and
$a(x,\xi)(1-g(\xi))$ as in the proof of Corollary \ref{Cor1}.
We remark that the operator $T$ defined by \eqref{FIOp2}
is the operator $P$ defined by \eqref{FIOp1}
with $\varphi(y,\xi)=y\cdot\xi-\psi(\xi)$ and
$a(x,y,\xi)=a(x,\xi)$ independent of $y$.
For the term $T_1$ corresponding to $a(x,\xi)(1-g(\xi))$,
we just apply Corollary \ref{Cor1}.
For the term $T_2$ corresponding to $a(x,\xi)g(\xi)$,
we have
\[
 T_2u(x)
=
\int_{\R^n} e^{i(x\cdot\xi+\psi(\xi))}a(x,\xi)g(\xi)\widehat u(\xi)\, d\xi
=
a(X,D_x)e^{i\psi(D_x)}g(D_x)u(x).
\]
The pseudo-differential operator $a(X,D_x)$ is $L^p$-bounded 
(see Kumano-go and Nagase \cite{KN})
and the Fourier multiplier $e^{i\psi(D_x)}g(D_x)$ is also $L^p$-bounded
by the Marcinkiewicz theorem (see Stein \cite{St}) 
since
$\abs{\partial^\alpha \p{e^{i\psi(\xi)}g(\xi)}}
\leq C_\alpha|\xi|^{-|\alpha|}$ for any multi-index $\alpha$.
The proof of Corollary \ref{Cor2} is complete.
\end{proof}

\end{document}